\documentclass[11pt, twoside, letter]{article}
\usepackage[margin = 1in]{geometry}
\usepackage{array}
\usepackage{subcaption}
\usepackage{amsthm}
\usepackage{amsmath}
\usepackage{amsfonts}
\usepackage{microtype}
\usepackage{graphicx}
\usepackage{comment}
\usepackage{booktabs} 
\newtheorem{theorem}{Theorem}[section]
\newtheorem{lemma}[theorem]{Lemma}
\newtheorem{proposition}[theorem]{Proposition}
\newtheorem{corollary}[theorem]{Corollary}
\theoremstyle{definition}
\newtheorem{definition}[theorem]{Definition}
\newtheorem{example}[theorem]{Example}

\usepackage{hyperref}

\theoremstyle{remark}
\newtheorem{remark}[theorem]{Remark}

\renewcommand{\vec}[1]{\mathbf{#1}}

\begin{document}

\title{Multiplicative Weights Update as a Distributed Constrained Optimization Algorithm: Convergence to Second-order Stationary Points Almost Always}

\author{Ioannis Panageas\\SUTD\\ioannis@sutd.edu.sg
\and Georgios Piliouras\\SUTD\\georgios@sutd.edu.sg
\and Xiao Wang\\SUTD\\xiao\_wang@sutd.edu.sg
}


\date{}
\maketitle

\begin{abstract}
Non-concave maximization has been the subject of much recent study in the optimization and machine learning communities, specifically in deep learning.
Recent papers  \cite{Ge15},~\cite{LPPSJR17}  and references therein indicate that first order methods work well and avoid saddle points.  Results as in \cite{LPPSJR17}, however, are limited to the \textit{unconstrained} case or for cases where the critical points are in the interior of the feasibility set, which fail to capture some of the most interesting applications. In this paper we focus on \textit{constrained} non-concave maximization. We analyze a variant of a well-established algorithm in machine learning called Multiplicative Weights Update (MWU) for the maximization problem $\max_{\mathbf{x} \in D} P(\mathbf{x})$, where $P$ is non-concave, twice continuously differentiable and $D$ is a product of simplices. We show that MWU converges almost always for small enough stepsizes to critical points that satisfy the second order KKT conditions.
We combine techniques from dynamical systems as well as taking advantage of a recent connection between Baum Eagon inequality and MWU \cite{PPP17}.
\end{abstract}

\section{Introduction}\label{sec:intro}

The interplay between the structure of saddle points and the performance of first order algorithms is a critical aspect of non-concave maximization. In  the \textit{unconstrained} setting, there have been many recent results indicating that gradient descent (GD) avoids strict saddle points with random initialization \cite{LPPSJR17}, (see also \cite{DP18} for the analogue in min-max optimization). Moreover by adding noise, it is guaranteed that GD converges to a local maximum in polynomial time (see \cite{Ge15}, \cite{ChiGe16} and references therein). By adding a non-smooth function in the objective (e.g., the indicator function of a convex set) it can be shown that there are stochastic first order methods that converge to a local minimum point in the \textit{constrained} case \cite{AZ1} \cite{AZ2}\cite{AZ3}\cite{AZ4} under the assumption of oracle access to the stochastic (sub)gradients. What is less understood is the problem of convergence to second order stationary points in \textit{constrained} optimization (under the weaker assumption that we do not have access to the subgradient of the indicator of the feasibility set; in other words when projection to the feasibility set is not a trivial task).
In the case of constrained optimization, we also note that the techniques of \cite{LPPSJR17} are not applicable in a straightforward way.

Non-concave maximization problems with saddle points/local optima on the boundary are very common. For example in game theory, it is typical  for a Nash equilibrium not to have full support (and thus to lie on the boundary of the simplex). In such cases, one natural approach is to use projected gradient descent, but computing the projection at every iteration might not be an easy task to accomplish. Several distributed, concurrent optimization techniques have been studied in such settings (\cite{KPT}, \cite{ABFH}, \cite{DP19}), however they are known to work only for very specific type of optimization problems, i.e., multilinear potential functions.
 Moreover, having saddle points/local optima on the boundary of a closed set that has (Lebesgue) measure zero compared to the full domain (e.g., simplex with $n$ variables has measure zero in $\mathbb{R}^n$) makes impossible to use as a black box the result in \cite{LPPSJR17} in which they make use of well-known Center-stable manifold theorem from the dynamical systems literature  (see Theorem \ref{thm:manifold} in the supplementary material).

\noindent
In this paper we focus on solving problems of the form
\begin{equation}\label{eq:maxopt}
\max_{\mathbf{x} \in D} P(\mathbf{x}),
\end{equation}
where $P$ is a non-concave, twice continuously differentiable function and $D$ is some compact set, which will be a product of simplices for our purposes, i.e., $D=\{(x_{ij})|x_{ij}\ge 0, \sum_{j=1}^M x_{ij}=1\textrm{ for all }1\leq i\leq N\}$, where $N,M$ are natural numbers. As a result, vector $\mathbf{x}$ can be also interpreted as a collection of $N$ probability distributions (having $N$ players), where each distribution $\mathbf{x}_i$ has support of size $M$ (strategies). For this particular problem (\ref{eq:maxopt}), one natural algorithm that is commonly used is the Baum-Eagon dynamics (\ref{eq:Baum_Eagon_dynamics}) (see the seminal paper by Baum and Eagon \cite{BE}) with many applications to inference problems, Hidden Markov Models (HMM) in particular (see also discussion in Section \ref{seq:applications}).
\begin{align}\label{eq:Baum_Eagon_dynamics}
x_{ij}^{t+1} = x_{ij}^{t} \frac{ \frac{\partial P}{\partial x_{ij}}\big|_{\mathbf{x}^t}}{\sum_s x^t_{is}\frac{\partial P}{\partial x_{is}}\big|_{\mathbf{x}^t}},
\end{align}
The denominator of the above fraction is for renormalization purposes (superscript $t$ indicates the iteration). It is clear that as long as $\mathbf{x}^t \in D$ then $\mathbf{x}^{t+1} \in D$.

Despite its power, Baum-Eagon dynamics has its limitations. First and foremost, the Baum-Eagon dynamics is not always well-defined; the denominator term $\sum_s x^t_{is}\frac{\partial P}{\partial x_{is}}\big|_{\mathbf{x}^t}$ must be non-zero at all times and moreover the fraction in equations (\ref{eq:Baum_Eagon_dynamics}) should always be non-negative. This provides a restriction to the class of functions $P$ to which the Baum-Eagon dynamics can be applied. Moreover, it turns out that the update rule of the Baum-Eagon dynamics is not always a diffeomorphism.\footnote{A function is called a diffeomorphism if it is differentiable and a bijection and its inverse is differentiable as well.} In fact, as we show even in simple settings (see section 2.3) the Baum-Eagon dynamics may not be even a homeomorphism or one-to-one. This counterexample disproves a conjecture by Stebe\cite{S}. Since the map is not even a local diffeomorphism one cannot hope to leverage the power of Center-stable manifold theorem to argue convergence towards local maxima.

To counter this, in this paper we  focus on multiplicative weights update algorithm (MWU) \cite{Arora1} which can be interpreted as an instance of Baum-Eagon dynamics in the presence of learning rates.
 Introducing learning rates gives us a lot of flexibility and will allow us to formally prove strong convergence properties which would be impossible without this adaptation.
 Assume that $\mathbf{x}^t$ is the $t$-th iterate of MWU, the equations of which can be described as follows:
\begin{align}\label{eq:MWUdynamics}
x_{ij}^{t+1} = x_{ij}^{t} \frac{1+\epsilon_i \frac{\partial P}{\partial x_{ij}}\big|_{\mathbf{x}^t}}{1+\epsilon_i\sum_s x^t_{is}\frac{\partial P}{\partial x_{is}}\big|_{\mathbf{x}^t}},
\end{align}
where $\epsilon_i$ the \textit{stepsize} (learning rate) of the dynamics. Intuitively (in game theory terms), for strategy profile (vector) $\tilde{\mathbf{x}} := (\tilde{\mathbf{x}}_1,...\tilde{\mathbf{x}}_N)$, each player $i$ that chooses strategy $j$ has utility to be $\frac{\partial P}{\partial x_{ij}} \big |_{\mathbf{x} = \tilde{\mathbf{x}}}$. We call a strategy profile $\mathbf{y} \in D$ a fixed point if it is invariant under the update rule dynamics (\ref{eq:MWUdynamics}).  It is also clear that the set $D$ is invariant under the dynamics in the sense that if $\mathbf{x}^t \in D$ then $\mathbf{x}^{t+1} \in D$ for $t \in \mathbb{N}$. This last observation indicates that MWU has the projection step for free (compared to projected gradient descent). We would also like to note that MWU can be computed in a distributed manner and this makes the algorithm more important for Machine Learning applications.

\paragraph{Statement of our results}



We will need the following two definitions (well-known in optimization literature, as applied to simplex constraints):

\begin{definition}[Stationary point] $\mathbf{x}^*$ is called a stationary point as long as it satisfies the first order KKT conditions for the problem (\ref{eq:maxopt}). Formally, it holds
\begin{equation}\label{eq:KKT}
\begin{array}{l}
\mathbf{x}^* \in D\\
x_{ij}^*>0 \Rightarrow \frac{\partial P}{\partial x_{ij}}(\mathbf{x}^*) = \sum_{j'}x_{ij'}^*\frac{\partial P}{\partial x_{ij'}}(\mathbf{x}^*)\\
x_{ij}^*=0 \Rightarrow \frac{\partial P}{\partial x_{ij}}(\mathbf{x}^*) \leq \sum_{j'}x_{ij'}^*\frac{\partial P}{\partial x_{ij'}}(\mathbf{x}^*).
\end{array}
\end{equation}
The stationary point is called strict if the last inequalities hold strictly.
\end{definition}

\begin{definition}[Second order stationary point] $\mathbf{x}^*$ is called a second order stationary point as long as it is a stationary point and moreover it holds that:
\begin{equation}\label{eq:KKTsecondorder}
\begin{array}{l}
\mathbf{y}^{\top}\nabla^2 P(\mathbf{x}^*)\mathbf{y}\leq 0.
\end{array}
\end{equation}
for all $\mathbf{y}$ such that $\sum_{j=1}^M y_{ij}=0$ (for all $1\leq i \leq N$) and $y_{ij}=0$ whenever $x^*_{ij}=0$, i.e., it satisfies the second order KKT conditions.
\end{definition}
Our main result are stated below:
\begin{theorem}[Avoid non-stationary]\label{thm:main2} Assume that $P$ is twice continuously differentiable in a set containing $D$. There exists small enough fixed stepsizes $\epsilon_i$ such that the set of initial conditions $\mathbf{x}^0$ of which the MWU dynamics (\ref{eq:MWUdynamics}) converges to fixed points that violate second order KKT conditions is of (Lebesgue) measure zero.
\end{theorem}

\noindent
The following corollary is immediate from Theorem \ref{thm:main2} and the Baum-Eagon inequality for rational functions (see Section \ref{sec:baumeagon}).
\begin{corollary}\label{cor:main2} Assume $\mu$ is a measure that is absolutely continuous with respect to the Lebesgue measure and $P$ is a rational function (fraction of polynomials) that is twice continuously differentiable in a set containing $D$, with isolated\footnote{A stationary point is isolated if there exists a neighborhood around it so that there is no other stationary point in that neighborhood.} stationary points. It follows that with probability one (randomness induced by $\mu$), MWU dynamics converges to second order stationary points.
\end{corollary}

\begin{remark}
It is obvious that when the learning rates $\epsilon_i=0$, MWU (\ref{eq:MWUdynamics}) is trivially the identity map. On the other hand, whenever the dynamics is well defined in the limit $\epsilon\rightarrow \infty$ (i.e. when $P$ is sufficiently well behaved, e.g. a polynomial with positive coefficients) this corresponds to the well known class of Baum-Eagon maps \cite{S}.
\end{remark}

We conclude our results by showing that it is unlikely that MWU dynamics converges fast to second (or even first) order stationary points when MWU is applied to solve problem (\ref{eq:maxopt}). The problem of finding first (resp. second) order stationary points are inherently connected with the problem of finding mixed (resp. pure) Nash equilibria in congestion games. Currently, no polynomial time algorithms are known for computing mixed Nash in congestion games (the problem lies between P and CLS\footnote{CLS is a computational complexity class that captures continuous local search. It lies on the intersection of the mores well studied classes of PLS and PPAD.}) \cite{Daskalakis:2011:CLS:2133036.2133098}, whereas computing pure Nash Nash equilibria even in linear congestion games,  is known to be PLS-complete \cite{Fabrikant:2004:CPN:1007352.1007445,Ackermann:2008:ICS:1455248.1455249}. The reductions between the problems is based on the fact that congestion games are potential games and hence (\ref{eq:MWUdynamics}) captures the behavior of self-interested learning agents playing a congestion game.




\paragraph{Our techniques}
The first step of the proof given in Section 3 is to prove that  MWU converges to fixed points for all rational functions and any possible set of learning rates (as long as the dynamics is well defined). The proof of this statement leverages recently discovered connections between MWU and the Baum-Eagon dynamics \cite{PPP17}. However, this does not even allow us to exclude very suboptimal fixed points (i.e. saddle points or even local minima) from having a positive region of attraction.

The other two steps of the proof work on weeding out the "bad" stationary points and showing that the set of initial conditions that converge to them is of measure zero. The key tool for proving that type of statements is the Center-stable manifold theorem \cite{LPPSJR17}. However, in order to leverage the power of the theorem we  first show in Theorem 2.3  that for small enough learning rates MWU is a diffeomorphism. The second and third step of the proof respectively is to show that fixed points that do not satisfy the first (resp. second) order stationary point conditions are unstable under MWU.

Even for the first step of the proof (lemma 3.1), we have to use ad-hoc techniques to deal with problems due to the constraints. Specifically, we start by projecting the domain $D$ to a subspace that is full dimensional (for example simplex of size $n$ is mapped to the Euclidean subspace of dimension $n-1$). Next, we show that non-first order  stationary points result to fixed points where the Jacobian of MWU has eigenvalue larger than $1$. Proving a similar statement for the fixed points that correspond to non-second order stationary fixed points (lemma 3.2) is the most technical part of the proof as we have to deal with the asymmetry of the resulting Jacobian. Nevertheless we manage to do so by using Sylvester's law of inertia and exploiting newly discovered decompositions for this class of matrices. Putting everything together results in our main theorem (Theorem 1.3).

\textbf{Notation} Throughout this article, $D$ is the product of $N$ simplices of size $M$ each,
$D=\{(x_{ij})|x_{ij}\ge 0, \sum_{j=1}^M x_{ij}=1\textrm{ for all }1\leq i\leq N\}, $ where we interpret $i$ as the index for the $N$ agents and $j$ the index of strategies $M$. We also use boldfaces to denote vectors, i.e., $\mathbf{x}$ and $[N]$ denotes $\{1,...,N\}$.

\section{Optimization with Baum-Eagon Algorithm}\label{sec:baumeagon}
In this section, we state the important result of Baum and Eagon providing a method to increase the value of a polynomial with nonnegative coefficients and (later generalized for) rational functions with nonzero denominators. The update rule defined by (\ref{eq:BEmap}) increases the value of the polynomial $P$ if the initial point is not a fixed point of Baum-Eagon dynamics.

\subsection{Baum-Eagon map}
Let $P$ be a polynomial with real positive coefficients and variables $x_{ij}$, $i=1,...,k,j=1,...,n_i$. Let $n=\sum_{i=1}^{k} n_i$. Let $D$ be the product of simplexes.
 Define $\vec{x}' := T(\mathbf{x})$ as the vector in $D$ with component $ij$ given by
 \begin{equation}\label{eq:BEmap}
 x'_{ij} = T(\vec{x})_{ij} :=\frac{x_{ij}\frac{\partial P}{\partial x_{ij}}}{\sum_{h=1}^{n_i}x_{ih}\frac{\partial P}{\partial x_{ih}}}.
 \end{equation}
\begin{theorem}[\cite{BE}]
Let $P(\{x_{ij}\})$ be a polynomial with non-negative coefficients homogeneous of degree $d$ in its variables $\{x_{ij}\}$. Let $\vec{x}=\{x_{ij}\}$ be any point of the domain $D=\{x_{ij}\ge 0, \sum_{j=1}^{n_i}x_{ij}=1,i=1,2,...,k,j=1,2,...,n_i\}$. For $\mathbf{x}=\{x_{ij}\}\in D$, let $T(\mathbf{x})=T(\{x_{ij}\})$ be the point of $D$ whose $i,j$ coordinate is
\begin{equation} \label{eq1}
 T(\vec{x})_{ij}=\frac{x_{ij}\frac{\partial P}{\partial x_{ij}}}{\sum_{h=1}^{n_i}x_{ih}\frac{\partial P}{\partial x_{ih}}}.
 \end{equation}
 Then $P(T(\mathbf{x}))>P(\mathbf{x})$ unless $T(\mathbf{x})=\mathbf{x}$.
\end{theorem}

\subsection{ Optimization for rational functions}
      According to \cite{GKNN}, one can define a Baum-Eagon dynamics for rational functions $R(\vec{x})=\frac{S_1(\vec{x})}{S_2(\vec{x})}$ with positive denominator so that the update rule of the Baum-Eagon dynamics increases the value of the rational function $R$ for any given vector $\vec{y}$ unless $\vec{y}$ is a fixed point. This can be done by starting with the Baum-Eagon map of the following polynomial: Let $\vec{y} \in D$ be an arbitrary point.
\[
Q_{\vec{y}}(\vec{x})=P_{\vec{y}}(\vec{x})+C_{\vec{y}}(\vec{x}),
\]
where
\[
P_{\vec{y}}(\vec{x})=S_1(\vec{x})- R(\vec{y})\cdot S_2(\vec{x})
\]
and
\[
C_{\vec{y}}(\vec{x})=N_{\vec{y}}(\sum _{ij}x_{ij}+1)^d,
\]
where $d$ is the degree of $P_{\vec{y}}(\vec{x})$ and $N_{\vec{y}}$ is a constant such that $P_{\vec{y}}(\vec{x})+C_{\vec{y}}(\vec{x})$ only has nonnegative coefficients.
\\
\\
It is proved in \cite{GKNN} that $R(T(\vec{y})) > R(\vec{y})$ along the Baum-Eagon dynamics (update rule $T$) induced by polynomial $Q_{\vec{y}}(\vec{x})$.



\subsection{Bad example on Baum-Eagon dynamics}
L. Baum has an unpublished result \cite{S} claiming that the Baum-Eagon map $T$ is a homeomorphism\footnote{A function is called a homeomorphism if it is continuous and a bijection and its inverse is continuous as well. Thus if a function is not a homeomorphism, then it is not a diffeomorphism.} of $D$ onto itself if and only if the polynomial $P$ can be expressed as a sum that \textit{contains} monomials of the form $c_{i,j}x_{i,j}^{w_{i,j}}$ for all $i=1,...,k,j=1,...,n_i$ where $c_{i,j}>0$ and $w_{i,j}$ is an integer greater than zero (this means that $P$ might also contain other terms, i.e, products of different variables).
But this condition is incorrect and we give a counter example below. We note that our example indicates that the Baum-Eagon dynamics does not satisfy the nice property of being a diffeomorphism.

For a special case, we focus on the map $\tau$ defined on a single simplex (with $n$ variables)
\[
\Delta_{n-1}=\{(x_1,...,x_n)|\sum_{i=1}^{n}x_i=1\},
\]
 and $\tau$ can be written as
\begin{equation}\label{eq2}
x_i'= \tau (\vec{x})_i:=\frac{x_i\frac{\partial P}{\partial x_i}}{\sum x_i\frac{\partial P}{\partial x_i}}
\end{equation}
The map defined in equation (\ref{eq2}) can be expressed as a composition of $\tau_1$ and $\tau_2$ defined in the following way:
\begin{align}
 \tau_1&:(x_1,...,x_n)\mapsto(x_1\frac{\partial P}{\partial x_1},...,x_n\frac{\partial P}{\partial x_n})\label{eq3}\\
 \tau_2&:(x_1\frac{\partial P}{\partial x_1},...,x_n\frac{\partial P}{\partial x_n})\mapsto \frac{1}{\sum_{i=1}^n x_i\frac{\partial P}{\partial x_i}}(x_1\frac{\partial P}{\partial x_1},...,x_n\frac{\partial P}{\partial x_n})\label{eq4}
\end{align}
Consider 1-dimensional simplex as an example (i.e, $n=2$), $\tau_1$ maps the simplex $\Delta_1$ to a curve and $\tau_2$ maps points on the curve back to $\Delta_1$ by scaling.
\begin{figure}[h]\label{figall}
\centering
\begin{subfigure}[b]{0.3\textwidth}
\includegraphics[width=\textwidth]{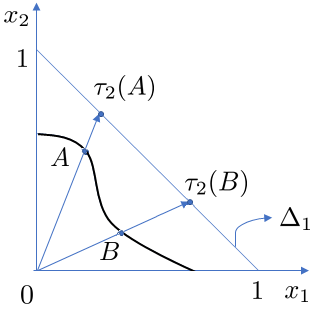}
\caption{$\tau=\tau_2\circ\tau_1$}
\label{fig1}
\end{subfigure}
\begin{subfigure}[b]{0.3\textwidth}
\includegraphics[width=\textwidth]{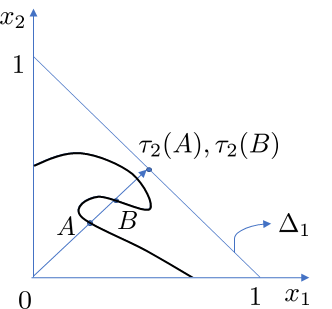}
\caption{$\tau=\tau_2\circ\tau_1$}. 
\label{fig2}
\end{subfigure}
\caption{Illustration}
\end{figure}
From Figure \ref{fig1}, we notice that a necessary condition for $\tau$ to be a homeomorphism is that the curve $\tau_1 (\Delta_1)$ (image of $\Delta_1$ under $\tau_1$, see thick, black curve in Figure \ref{fig1}, \ref{fig2}) does not cross twice (or more times) any line that passes through the origin and has slope non-negative (see also Figure \ref{fig2}). A necessary condition for $\tau$ to be a homeomorphism is that $\tau$ must be one to one. In 1-dimensional case, the ratio
\[
k=x_1\frac{\partial P}{\partial x_1}/x_2\frac{\partial P}{\partial x_2}
\]
must be monotone with respect to $x_1$. The following example is a polynomial that satisfies Baum's condition, however it holds that function $k$ is not monotone with respect to $x_1$.
\begin{example}
Suppose $P=x_1+x_1^7x_2+x_2^7$, then

\begin{align*}
x_1\frac{\partial P}{\partial x_1}&=x_1+7x_1^7x_2 \\
x_2\frac{\partial P}{\partial x_2}&=x_1^7x_2+7x_2^7
\end{align*}
As it is shown in Figure \ref{CountEx}, the ratio  $k=x_1\frac{\partial P}{\partial x_1}/x_2\frac{\partial P}{\partial x_2}$ is not monotone with respect to $x_1$. So the Baum-Eagon map is not one to one implying that it is not a homeomorphism.
\begin{figure}[h]
\centering
\includegraphics[width=0.5\textwidth]{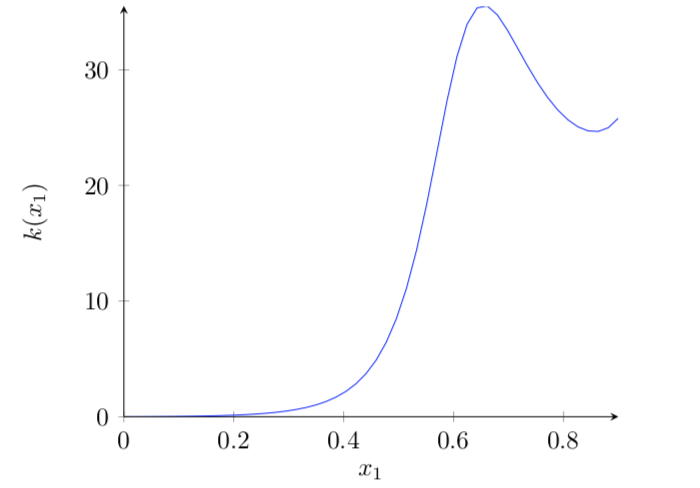}
\caption{Non-monotonicity of $k(x_1)$}
\label{CountEx}
\end{figure}

\end{example}

\subsection{Baum-Eagon map of $\sum_{i,j}x_{ij}+\epsilon P$}
Let $P$ be a twice continuously differentiable function on the product of simplexes $D$. The update rule of the Baum-Eagon dynamics for the function $Q := \sum_{i,j}x_{ij}+\epsilon P$ (as defined in (\ref{eq:Baum_Eagon_dynamics})) is a diffeomorphism for $\epsilon$ sufficiently small (we note that Baum-Eagon dynamics for $Q$ coincides with the MWU dynamics for $P$, see Equations (\ref{eq:MWUdynamics})). This is what next theorem captures.

\begin{theorem}\label{thm:diffeomorphism}
For any twice continuously differentiable function $P$, there exists a positive number $\delta$ depending on $P$, such that for any $\epsilon<\delta$, the Baum-Eagon map applied to $Q=\sum_{ij}x_{ij}+\epsilon P$ is a diffeomorphism.
\end{theorem}
\begin{proof}
Firstly, we prove that the Baum-Eagon map of $Q$ is a local diffeomorphism.
For a fixed $i$, denote
\[
T(\vec{x})_{ij}=\frac{x_{ij}+\epsilon x_{ij}\frac{\partial P}{\partial x_{ij}}}{\sum_j x_{ij}+\epsilon\sum_j x_{ij}\frac{\partial P}{\partial x_{ij}}}.
\]

Since the roots of the characteristic polynomial of a matrix vary continuously as a function of coefficients (see Theorem VI.1.2 in \cite{RB}), let $J_{\epsilon}$ be the Jacobian of the function $T(\vec{x})$, i.e., of the update rule of the Baum-Eagon dynamics induced by function $Q=\sum_{ij} x_{ij}+\epsilon P$ (note that $T$ coincides with the MWU dynamics for function $P$ with same stepsize $\epsilon$ (i.e., same learning rates)). The determinant $|J_{\epsilon}|$ is continuous with respect to $\epsilon$. When $\epsilon\rightarrow 0$, it holds that $|J_{\epsilon}|\rightarrow 1$ at each point $p \in D$ where the Jacobian is computed, thus for each point $p\in D$ there exists $\epsilon_p$, such that for all $\epsilon<\epsilon_p$, we get that $|J_{\epsilon}(p)|>1/2$.

Since the determinant is also continuous with respect to points in $D$, for $\epsilon_p$, there is a neighborhood of $p$, denoted as $U(p,\epsilon_p)$, such that for all $\vec{x} \in U(p,\epsilon_p)$, $|J_{\epsilon_p}(\vec{x})|>1/2$. Thus we have obtained an open cover of $D$, which is $\bigcup_{p\in D}U(p,\epsilon_p)$. Since $D$ is compact, there is a finite subcover of $\bigcup_{p\in D}U(p,\epsilon_p)$, denoted as $\bigcup_{i=1}^nU(p_i,\epsilon_{p_i})$. Then the minimum of $\{\epsilon_{p_i}\}$ gives the $\delta$ in the lemma.

To prove that the Baum-Eagon map $T$ of $Q$ is a global diffeomorphism, one needs Theorem 2 in \cite{HCW}. Since $T$ is proper (preimage of compact set is compact) and $D$ is simply connected and path connected, we conclude that $T$ is a homeomorphism on $D$ (we suggest the reader to see the supplementary material for all the missing definitions).
\end{proof}
\vspace{-5pt}
\begin{remark} The above theorem essentially can be generalized for different stepsizes (learning rates) $\epsilon$ for each player. The idea is that we should apply the same techniques on the function $\sum_{i=1}^N \frac{1}{\epsilon_i} \sum_{j=1}^M x_{ij}+ P$.
\end{remark}

\section{Convergence Analysis of MWU for Arbitrary Functions}
In this section we provide the proof of Theorem~\ref{thm:main2}. As has already been proven in previous section (Theorem \ref{thm:diffeomorphism}), the update rule of the MWU dynamics is a diffeomorphism for appropriately small enough learning rates. Following the general framework of \cite{LPPSJR17}, we will also make use of the Center-stable manifold theorem (Theorem \ref{thm:manifold}). The challenging part technically in this paper is to prove that every stationary point $\mathbf{x}$ that is not a local maximum has the property that the Jacobian of the MWU dynamics computed at $\mathbf{x}$ has a repelling direction (eigenvector).

\subsection{Equations of the Jacobian at a fixed point and projection}\label{sec:eqJac}
We focus on multiplicative weights updates algorithm. Assume that $\mathbf{x}^t$ is the $t$-th iterate of MWU. Recall the equations:
\begin{equation}\label{eq:MWU}
x_{ij}^{t+1}=x_{ij}^t\frac{1+\epsilon_i\frac{\partial P}{\partial x_{ij}}|_{\vec{x} = \mathbf{x}^t}}{1+\epsilon_i\sum_s x_{is}^t\frac{\partial P}{\partial x_{is}}|_{\vec{x}=\mathbf{x}^t}}
\end{equation}
where $\epsilon_i$ the stepsize of the dynamics. Let $T : D \to D$ be the update rule of the MWU dynamics (\ref{eq:MWU}). Fix indexes $i,i'$ for players and $j,s$ for strategies. Set $S_i = 1+\epsilon_i\sum_{j'}x_{ij'}\frac{\partial P}{\partial x_{ij'}}$. The equations of the Jacobian look as follows:
\begin{align*}
\frac{\partial T_{ij}}{\partial x_{ij}}&=\frac{1+\epsilon_i\frac{\partial P}{\partial x_{ij}}}{S_i}
+\frac{x_{ij}}{S_i^2}\left(\epsilon_i\frac{\partial^2P}{\partial x_{ij}^2}\cdot S_i-\epsilon_i\left(1+\epsilon\frac{\partial P}{\partial x_{ij}}\right)
\cdot\left(\frac{\partial P}{\partial x_{ij}}+x_{ij}\frac{\partial^2P}{\partial x_{ij}^2}+\sum_{s\ne j}x_{is}\frac{\partial P}{\partial x_{is}\partial x_{ij}}\right)\right)
\end{align*}
for all $i\in[N],$ $j\in [M],$
\begin{align*}
\frac{\partial T_{ij}}{\partial x_{is}}&=\frac{x_{is}}{S_i^2}\left(\epsilon_i\frac{\partial^2P}{\partial x_{ij}\partial x_{is}}\cdot S_i
-\epsilon_i\left(1+\epsilon_i\frac{\partial P}{\partial x_{ij}}
\cdot\left(\frac{\partial P}{\partial x_{is}}+x_{is}\frac{\partial^2 P}{\partial x_{is}^2}+\sum_{j'\ne s}x_{ij'}\frac{\partial^2P}{\partial x_{ij'}\partial x_{is}}\right)\right)\right)
\end{align*}
for all $i\in[N],$ $j,s\in [M],$ $j\ne s$ and
\begin{align*}
\frac{\partial T_{ij}}{\partial x_{i's}}&=\frac{x_{ij}}{S_i^2}\left(\epsilon_i\frac{\partial^2P}{\partial x_{i's}\partial x_{ij}}\cdot S_i
-\epsilon_i\left(1+\epsilon_i\frac{\partial P}{\partial x_{ij}}\right)\cdot\left(\sum_{j'}x_{ij'}\frac{\partial^2P}{\partial x_{ij'}\partial x_{i's}}\right)\right)
\end{align*}
for all $i,i'\in [N],$ $j,s\in [M],$ with $i\ne i'$.

Let $\mathbf{y}$ to be a fixed point of MWU dynamics. We define the projected MWU mapping to be the function $T_{\mathbf{y}}$ by removing \textit{one} variable $j \in [M]$ for each player $i \in [N]$ (i.e, $x_{ij}$) such that $y_{ij}>0$. We also define $D_{\mathbf{y}}$ to be the projection of $D$ in the same way. Now the mapping is $T_{\mathbf{y}} : \mathcal{S} \to \mathcal{S}$ for $\mathcal{S} \subset \mathbb{R}^{NM-M}$ is still a diffeomorphism where $\mathcal{S}$ is an open set that contains $D_{\mathbf{y}}$. We define the corresponding Jacobian (called \textit{projected Jacobian}) to be the submatrix by removing rows and columns that correspond to variables $x_{ij}$ that were removed.

\subsection{Stability and proof of Theorem \ref{thm:main2}}

We prove the following important lemma that characterizes (partially) the unstable fixed points (meaning the spectral radius of the Jacobian computed at the fixed point is greater than one) of the MWU dynamics and relates them to the stationary points.

\begin{lemma}[Non first order stationary points are unstable]\label{lem:unstable} Let $\mathbf{y}$ be a fixed point of MWU dynamics that violates the first order KKT conditions (is not a first order stationary point). It holds that the projected Jacobian computed at $\mathbf{y}$ (formally now is the projected point $\mathbf{y} \in D_{\mathbf{y}}$) has an eigenvalue with absolute value greater than one.
\end{lemma}
\begin{proof} Since $\mathbf{y}$ is not a stationary point, there exist $i,j$ and so that $y_{ij}=0$ but $\frac{\partial P}{\partial x_{ij}} \big |_{\mathbf{x} = \mathbf{y}} > \sum_{j'} y_{ij'} \frac{\partial P}{\partial x_{ij'}} \big |_{\mathbf{x} = \mathbf{y}}$. The projected Jacobian computed at $\mathbf{y}$ has the property that for variable $x_{ij}$, the corresponding row
has entries zeros, apart from the corresponding diagonal entry that is $\frac{1+\epsilon_i \frac{\partial P}{\partial x_{ij}}}{1+\epsilon_i\sum_{j'}x_{ij'}\frac{\partial P}{\partial x_{ij'}}}>1$ (from the definition of stationary point). Since the projected Jacobian has as eigenvalue $\frac{1+\epsilon_i \frac{\partial P}{\partial x_{ij}}}{1+\epsilon_i\sum_{j'}x_{ij'}\frac{\partial P}{\partial x_{ij'}}}$ the claim follows.
\end{proof}

The following technical lemma gives a full characterization among the unstable fixed points of MWU dynamics and the second order stationary points (local maxima). This lemma is more challenging than the stability analysis in \cite{LPPSJR17} due to the fact that we have constraints on simplex.
\begin{lemma}[Non second order stationary points are unstable]\label{lem:unstable2} Let $\mathbf{x}^*$ be a fixed point of MWU dynamics that is a stationary point (satisfies first order KKT conditions) and violates the second order KKT conditions (is not a second order stationary point). It holds that the projected Jacobian computed at $\mathbf{x}^*$ (formally now is the projected point $\mathbf{x}^* \in D_{\mathbf{x}^*}$) has an eigenvalue with absolute value greater than one.
\end{lemma}
\begin{proof}
Because of Lemma \ref{lem:unstable} we may assume that $\mathbf{x}^*$ in the interior of $D$ (all coordinates are positive). Set $S_i = 1+\epsilon_i\sum_{j'=1}^M x^*_{ij'}\frac{\partial P}{\partial x_{ij'}}\big |_{\mathbf{x} = \mathbf{x}^*}$ for $i \in [N]$.
Set
\[
D_{xs}=
\]
\[
\left(
\begin{array}{ccccccc}
\frac{\epsilon_1 x^*_{11}}{S_1} & &\text{\Large 0}  & & & &\\
 & \ddots & & & &\text{\Huge 0}& \\
 \text{\Large 0}  & &\frac{\epsilon_1 x^*_{1M}}{S_1} & & & &\\
 & & &\ddots& & &\\
 & & & &\frac{\epsilon_N x^*_{N1}}{S_N}& &\text{\Large 0}\\
 &\text{\Huge 0} & & & &\ddots &\\
 & & & &\text{\Large 0} & &\frac{\epsilon_N x^*_{NM}}{S_N}
\end{array}
\right)
\]
and
\[
D_{xx}=
\left(
\begin{array}{ccccccc}
x^*_{11} & \cdots & x^*_{1M} & & & &\\
\vdots & & \vdots & & &\text{\Huge 0}& \\
x^*_{11}& \cdots & x^*_{1M} & & & &\\
&&&\ddots&&&\\
& & & &x^*_{N1}&\cdots & x^*_{NM}\\
& \text{\Huge 0} & & &\vdots& &\vdots\\
& & & &x^*_{N1}& \cdots & x^*_{NM}
\end{array}
\right)
\]
where $D_{xs}$ is a diagonal matrix with positive diagonal entries and $D_{xx}$ has rank $N$. The Jacobian (not projected) of MWU dynamics computed at $\mathbf{x}^*$ can be expressed in a compact form (see Section \ref{sec:eqJac} for the equations of the Jacobian) as
\begin{equation}
I + D_{xs}(I - D_{xx}) \nabla^2 P(\mathbf{x}^*) = I + D_{xs}\nabla^2 P(\mathbf{x}^*) - D_{xs} D_{xx} \nabla^2 P(\mathbf{x}^*),
\end{equation} where $I$ denotes the identity matrix (in particular of size $NM$ in the aforementioned expression).
 Observe that if $\mathbf{x}^*$ violates the second order KKT conditions it means that for the symmetric matrix $\nabla^2 P(\mathbf{x}^*)$ there is a vector $\mathbf{z}$ orthogonal to all ones vector (for each player) such that $\mathbf{z}^{\top}\nabla^2 P(\mathbf{x}^*) \mathbf{z}$ is positive. 
 Moreover, it holds that the Jacobian has the same eigenvalues with the matrix $I + D_{xs}^{1/2}\nabla^2 P(\mathbf{x}^*)D_{xs}^{1/2} - D_{xs}^{1/2} D_{xx} \nabla^2 P(\mathbf{x}^*)D_{xs}^{1/2}$ (since $D_{xs}$ is positive diagonal). Let us define $\tilde{Q} := D_{xs}^{1/2}\nabla^2 P(\mathbf{x}^*)D_{xs}^{1/2} - D_{xs}^{1/2} D_{xx} \nabla^2 P(\mathbf{x}^*)D_{xs}^{1/2}$. 
 To finish the proof, observe that for the vector $\mathbf{z}' := D_{xs}^{-1/2}\mathbf{z}$ it holds that $\mathbf{z}'^{\top} \tilde{Q} \mathbf{z}' = \mathbf{z}^{\top} \nabla^2 P(\mathbf{x}^*) \mathbf{z}>0$ (which is positive). Therefore $\tilde{Q}$ is a symmetric matrix with at least one positive eigenvalue, hence the Jacobian has an eigenvalue which is greater than one. It is easy to see that this is also an eigenvalue of the projected Jacobian and the claim follows.
\end{proof}

We can now prove our second main Theorem \ref{thm:main2}.

\begin{proof}[Proof of Theorem \ref{thm:main2}] As long as we establish the idea of projecting the Jacobian, then the proof follows the lines of work of \cite{MPP15}, \cite{LPPSJR17} and is rather generic. We shall show that the set of initial conditions so that MWU dynamics converges to unstable fixed points (meaning that the spectral radius of the Jacobian computed at the fixed point is greater than one) is of measure zero and then by Lemma \ref{lem:unstable}, the proof follows. Let $\mathbf{y}$ be an unstable fixed point of the MWU (as a dynamical system) with update rule a function $T_{\mathbf{y}}: \mathcal{S} \to \mathcal{S}$. For such unstable fixed point $\mathbf{y}$, there is an associated open neighborhood $B_{\mathbf{y}} \subset \mathcal{S}$ promised by the Stable Manifold Theorem \ref{thm:manifold}.

Define $W_{\mathbf{y}}=\{\mathbf{x}^0 \in D_{\mathbf{y}}: \lim_{t \to \infty}  \mathbf{x}^t = \mathbf{y}\}$. Fix a point $\mathbf{x}^0 \in W_{\mathbf{y}}$. Since $\mathbf{x}^k \to \mathbf{y} $, then for some non-negative integer $K$ and all $t\ge K$,  $T_{\mathbf{y}} ^t (\mathbf{x}^0) \in B_{\mathbf{y}}$ ($T_{\mathbf{y}}^t$ denotes composition of $T_{\mathbf{y}}$ $t$ times). We mentioned above that $T_{\mathbf{y}}$ is a diffeomorphism in $\mathcal{S}$. By Theorem \ref{thm:manifold}, $Q_{\mathbf{y}} := \cap_{k=0}^\infty T^{-k}_{\mathbf{y}} ( B_{\mathbf{y}})$  is a subset of the local center-stable manifold  which has co-dimension at least one, and $Q_{\mathbf{y}}$ is thus measure zero.

Finally, $T_{\mathbf{y}}^{K}(\mathbf{x}^0) \in Q_{\mathbf{y}}$ implies that $\mathbf{x}^0 \in T_{\mathbf{y}}^{-K} ( Q_{\mathbf{y}})$. Since $K$ is unknown we union over all non-negative integers, to obtain $\mathbf{x}^0 \in \cup_{j=0}^\infty T_{\mathbf{y}}^{-j} (Q_{\mathbf{y}})$. Since $\mathbf{x}^0$ was arbitrary, we have shown that $W_{\mathbf{y}} \subset  \cup_{j=0}^\infty T_{\mathbf{y}}^{-j} (Q_{\mathbf{y}})$. Using Lemma 1 of page 5 in \cite{LPPSJR17} and that countable union of measure zero sets is measure zero, $W_{\mathbf{y}}$ has measure zero. The claim follows since by mapping $W_{\mathbf{y}}$ to the set $W$ (which is defined by padding the removed variables), then $W$ is the set of initial conditions that MWU dynamics converges to $\mathbf{y}$ and is of measure zero in $D$.
\end{proof}

\subsection{On the speed of convergence}

In this section we argue about the limitations of any algorithm that aims at solving maximization problem subject to simplex constraints (even for polynomial objectives), i.e., problem (\ref{eq:maxopt}). We conclude that it is unlikely that MWU dynamics (or any other algorithm) converges in polynomial time to a local maximum (for problem (\ref{eq:maxopt})). In fact, as we will show  providing a polynomial time algorithm for finding even first order stationary points for an arbitrary polynomial function is at least as hard as computing Nash equilibria for general congestion games, a problem for which no polynomial time algorithm is known and whose time complexity lies in CLS \cite{Daskalakis:2011:CLS:2133036.2133098}. Computing second order stationary points even for general bilinear functions, specifically even for function of the form $f(\mathbf{x})= \sum_{i, i', i\neq i'} \sum_{j, j'} a_{ii'jj'} x_{ij} x_{i'j'} + \sum_i \sum_j b_{ij}x_{ij}$ is strongly connected with the problem finding pure Nash equilibria even in linear congestion games that is known to be PLS-complete  \cite{Fabrikant:2004:CPN:1007352.1007445,Ackermann:2008:ICS:1455248.1455249}.

Specifically, it suffices to focus on a special class of congestion games which are called threshold games. These are congestion games in which the set of resources $R$ is divided into two disjoint subsets $R_{in}$ and $R_{out}$. The set ${R}_{out}$ contains a resource $r_i$ for
every player $i \in N$ . This resource has a fixed delay $T_i$ called the threshold of player $i$. Each
player $i$ has exactly two strategies:  a strategy $S_i^{out}=\{r_i\}$
with $r_i \in R_i^{out}$, and a strategy $S_i^{in}\subseteq R_{in}$.
 Agent $i$ prefers strategy $S_i^{in}$ to  strategy $S_i^{out}$ if the total cost of playing $S_i^{in}$ is smaller than the threshold cost $T_i$. Quadratic threshold games are a subclass of threshold games in which the
set $R_{in}$ contains exactly one resource $r_{ii'}$ for every unordered pair of players $\{i, i'\} \subset N$. For every player $i \in N$ of a quadratic threshold game, his strategy set  $S_{in} = \{r_{ii'} | i'\in N, j i'\neq i\}$. Without loss of generality let any resource $r_{ii'}$ have a linear delay function of the form $c_{ii'}(k) = a_{ii'}k$ with $a_{ii'} > 0$. Furthermore, all thresholds can be assumed to be positive.
 \cite{Ackermann:2008:ICS:1455248.1455249} proves that computing a Nash equilibrium of a quadratic threshold game with nondecreasing delay functions is PLS-complete.

\begin{theorem}
 Finding a first-order stationary point for a general polynomial function $f$ is at least as hard as computing a Nash equilibrium for general congestion games.
Let $f(\mathbf{x})= \sum_{i, i', i\neq i'} \sum_{j, j'} a_{ii'jj'} x_{ij} x_{i'j'} + \sum_i \sum_j b_{ij}x_{ij}$, where for all $i$, $\sum x_{ij}=1$. Finding a
second-order stationary point of $f(x)$ is at least as hard as computing a pure Nash equilibrium in a generic quadratic threshold game.
\end{theorem}

\begin{proof}
Firstly, any first order stationary point of the expected value of the potential is a Nash equilibrium, since the gradient of the potential corresponds to the vector of deviating payoffs for all agents and all strategies. Thus. first order stationarity implies that only strategies that give maximal payoff are played with positive probability, i.e. the strategy is a Nash equilibrium.
The expected value of the potential function of a quadratic threshold congestion games is a bilinear function. This is trivially true since each resource can only be used by at most two agents.  Specifically the expected value of the potential function of the game when each agent $i$ is using mixed strategy $(x_{i S_i^{in}},x_{i S_i^{out}})$ is equal to $\sum_{i\in N} x_{i S_i^{in}}T_i + \sum_{i\in N} x_{i S_i^{out}}\sum_{i'\neq i} a_{ii'} +\sum_{i,i',i\neq i'}  x_{i S_i^{out}}x_{i' S_{i'}^{out}} a_{ii'}$. By the genericity assumption we can assume that the number of fixed points of MWU are finite and isolated, e.g. \cite{KPT}.
  If this Nash equilibrium is pure then we are done. Suppose not, in which case there exist some agents that play mixed strategies with support equal to $2$.
Since the potential is a bilinear function it can be computed without error using its gradient and Hessian via Taylor expansion. Second order stationarity now implies that for any coordinated set of deviations of two  of the randomizing agents the potential can still not improve.
Consider the continuum of strategy profiles $(\zeta_i, x_{-i})$ where $i$ was a randomizing agent  that now deviates and plays strategy $S_i^{in}$ with arbitrary probability $\zeta_i \in [0,1]$. Since the original strategy profile $x$ is a NE, agent $i$ is still indifferent between his two actions. As we have argued any profile that exactly two randomizing agents deviate does not affect the value of the expected potential for so the value of the potential does not change. So, even if agent $i'$ was to deviate to strategy  $\zeta_i' \in [0,1]$, the value of the potential at  $(\zeta_i,\zeta_i' ,  x_{-i,i'})$ cannot be higher that its value at $(\zeta_i, x_{-i})$ and $x$. So, none of the randomizing agents at any strategy profile $(\zeta_i, x_{-i})$ can profit by deviating. Each point on the line segment $(\zeta_i, x_{-i})$  with $\zeta_i \in [0,1]$ is a stationary point of MWU, and we reach a contradiction to our genericity assumption. Thus, the second order stationary point of the potential is a pure Nash.
\end{proof}

\section{Applications}
\label{seq:applications}
\begin{figure}[h]
\centering
\includegraphics[width=0.38\textwidth]{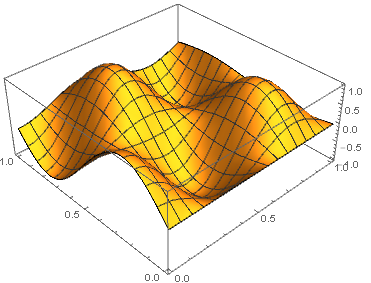}
\caption{Landscape of non-concave function $cos(8x)sin(6y)$.}
\label{fig:nonconvex}
\end{figure}

One application of Baum-Eagon algorithm is parameter estimation via maximum likelihood. Suppose that $X_1,...,X_n$ are samples from a population with probability density function $f(x|\theta_1,...,\theta_k)$, the likelihood function is defined by
\[
L(\theta|\mathbf{x})=L(\theta_1,...,\theta_k|x_1,...,x_n)=\prod^n_{i=1}f(x_i|\theta_1,...,\theta_k).
\]
Maximum likelihood estimator has many applications in machine learning and statistics (e.g., regression) and when is consistent, the problem of estimation boils down to maximizing the likelihood function. This can be achieved via the E-M algorithm based on the Baum-Eagon inequality. For example, the estimation of the parameters of hidden Markov models (motivated by real world problems, see \cite{GKNN} for an example on speech recognition) result in the maximization of rational functions over a domain of probability values. The rational functions are conditional likelihood functions of parameters $\theta=(\theta_1,...,\theta_k)$. The Baum-Eagon dynamics is used to estimate the parameters of hidden Markov models. Our main result indicates that MWU dynamics should be used for the optimization part as MWU has some nice properties (well-defined, update rule is a diffeomorphism, avoids non-stationary points) in which Baum-Eagon dynamics might not have.

Below we provide a pictorial illustration of MWU dynamics applied to a non-concave function (not rational). The function we consider is $P(x,y) = cos(8x)sin(6y)$ and we want to optimize it over $R = \{(x,y):0\leq x\leq1, 0\leq y\leq 1\}$ (see Figure \ref{fig:nonconvex} for the landscape). The aforementioned instance is captured by our model for $N=M=2$, in which we have essentially projected the space by using one variable for each player (for player one, the second variable is $1-x$ and for player two is $1-y$). The equations of MWU dynamics boil down to the following:
\begin{equation}\label{eq:Pxy}
\begin{array}{l}
x^{t+1} = x^t \frac{1+\epsilon(-8\sin(8x)\sin(6y))}{1+\epsilon x\cdot (-8\sin(8x)\sin(6y))+ \epsilon (1-x)\cdot (8\sin(8x)\sin(6y))}\\
y^{t+1} = y^t \frac{1+\epsilon(6\cos(8x)\cos(6y))}{1+\epsilon y\cdot (6\cos(8x)\cos(6y))+ \epsilon (1-y)\cdot (-6\cos(8x)\cos(6y))}
\end{array}
\end{equation}

We demonstrate in Figure \ref{fig:dynamics} the ``vector field" of MWU dynamics (because it is a discrete time system it is not precisely vector field, at point $(x,y)$ we plot a vector with direction $T(x,y)-(x,y)$, where $T$ is the update rule of dynamics (\ref{eq:Pxy})). The three dots indicate the local maxima of $P$ and the rest of the points do not satisfy the second order KKT conditions. We see that MWU dynamics avoids those points that do not satisfy the second order KKT conditions (avoids those that are not local maxima).

\begin{figure}[h!]
\centering
\includegraphics[width=0.44\textwidth]{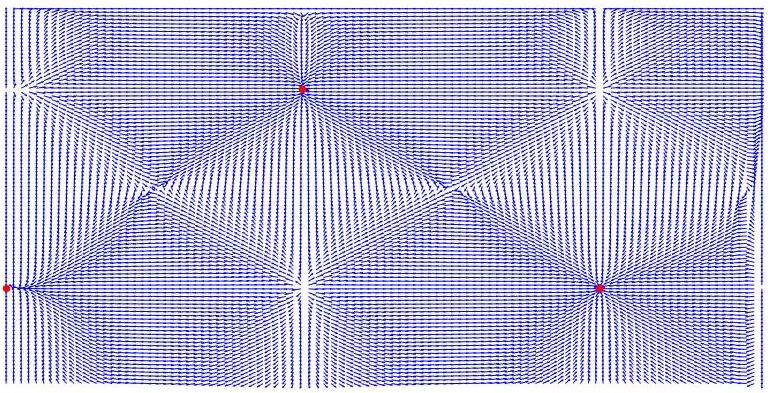}
\caption{Vector field of MWU dynamics in the case of non-concave function $cos(8x)sin(6y)$. Only local maxima (red dots) have positive regions of attraction.}
\label{fig:dynamics}
\end{figure}


\bibliography{example_paper}
\bibliographystyle{plain}
\appendix
\section{Stable manifold theorem}
\begin{theorem}[Center-stable manifold theorem, III.7  \cite{shub1987global}]\label{thm:manifold}
	Let $x^*$ be a fixed point for the $C^r$ local diffeomorphism $g: \mathcal{X} \to \mathcal{X}$. Suppose that $E = E_s \oplus E_u$, where $E_s$ is the span of the eigenvectors corresponding to eigenvalues of magnitude less than or equal to one of $D g(x^*)$, and $E_u$ is the span of the eigenvectors corresponding to eigenvalues of magnitude greater than one of $D g(x^*)$\footnote{Jacobian of function $g$.}. Then there exists a $C^r$ embedded disk $W^{cs}_{loc}$ of dimension $dim(E^s)$ that is tangent to $E_s$ at $x^*$ called the \emph{local stable center manifold}.  Moreover, there exists a neighborhood $B$ of $x^*$, such that $g(W^{cs}_{loc} ) \cap B \subset W^{cs} _{loc}$, and $\cap_{k=0}^\infty g^{-k} (B) \subset W^{cs}_{loc}$.
\end{theorem}

\section{Preliminaries on Topology}
This section provides fundamentals used in the proof of Theorem \ref{thm:diffeomorphism}. For more information on proper maps and the fundamental group, see  \cite{HCW} and \cite{AH}. 
\begin{definition}
$X$ is a \emph{Hausdorff space} if for every pair of points $p,q\in X$, there are disjoint open subsets $U,V\subset X$ such that $p\in U$ and $q\in V$.
\end{definition}
\begin{definition}
Let $X$ and $Y$ are topological spaces. A map from $X$ to $Y$, denoted $f:X\rightarrow Y$, is called \emph{proper} if the inverse of each compact subset of $Y$ is a compact subset of $X$.
\end{definition}
\begin{example}
Let $X$ be a compact space and $Y$ be a Hausdorff space. And suppose $f:X\rightarrow Y$ is continuous. Then $f$ is a proper map. Furthermore, if $D$ is compact subset of $\mathbb{R}^n$, then a continuous map $f:D\rightarrow D$ is proper.
\end{example}
\begin{definition}
A topological space $X$ is \emph{connected} if there are no disjoint open subsets $U,V\subset X$, such that $U\cup V=X$.
\end{definition}
\begin{definition}
A \emph{path} from a point $x$ to a point $y$ in a topological space $X$ is a continuous function $f:[0,1]\rightarrow X$ with $f(0)=x$ and $f(1)=y$. The space $X$ is said to be \emph{path-connected} if there exists a path joining any two points in $X$. A \emph{homotopy} of paths in $X$ is a family $f_t:[0,1]\rightarrow X$, $0\le t\le 1$, such that
\begin{itemize}
\item The endpoints $f_t(0)=x_0$ and $f_t(1)=x_1$ are independent of $t$.
\item The associated map $F:[0,1]\times [0,1]\rightarrow X$ defined by $F(s,t)=f_t(s)$ is continuous.
\end{itemize}
When two paths $f_0$ and $f_1$ are connected in this way by a homotopy $f_t$, they are said to be \emph{homotopic}. The notation for this is $f_0\simeq f_1$.
\end{definition}
\begin{proposition}
The relation of homotopy on paths with fixed endpoints in any space is an equivalence relation.
\end{proposition}
The equivalence class of a path $f$ under the equivalence relation of homotopy is denoted $[f]$ and called the \emph{homotopy class} of $f$.
\\
Given two paths $f,g:[0,1]\rightarrow X$ such that $f(1)=g(0)$, there is a \emph{product path} $f\cdot g$ that traverses first $f$ and then $g$, defined by the formula
\[
f\cdot g(s)=\left\{
\begin{array}{ll}
f(2s) , \ \ \ \ \ [1\le s\le\frac{1}{2}]
\\
g(2s-1), \ \ \ [\frac{1}{2}\le s\le 1]
\end{array}
\right.
\]
\begin{definition}
The paths with the same starting and ending point are called \emph{loops}, and the common starting and ending point is called the \emph{basepoint}. The set of all homotopy classes $[f]$ of loops $f:[0,1]\rightarrow X$ at the basepoint $x_0$ is denoted $\pi_1(X,x_0)$.
\end{definition}
\begin{proposition}
$\pi_1(X,x_0)$ is a group with respect to the product $[f][g]=[f\cdot g]$. 
\end{proposition}
And the group $\pi_1(X,x_0)$ is called the \emph{fundamental group}.
\begin{definition}
A topological space $X$ is \emph{simply connected} if it is path-connected and has trivial fundamental group.  
\end{definition}
\begin{example}
The space $D$ that is a product of simplexes is simply-connected. 
\end{example}
The following theorem (used in the proof of Theorem \ref{thm:diffeomorphism}) gives a sufficient and necessary condition under which a local homeomorphism $f:X\rightarrow Y$ becomes a global homeomorphism.
\begin{theorem}[Theorem 2, \cite{HCW}]
Let $X$ be path-connected and $Y$ be simply-connected Hausdorff spaces. A local homeomorphism $f:X\rightarrow Y$ is a global homeomorphism of $X$ to $Y$ if and only if the map $f$ is proper.
\end{theorem}

\end{document}